\documentclass[ar4paper,12pt,english,openany]{article}
\usepackage[latin1]{inputenc}
\usepackage[T1]{fontenc}
\usepackage[francais]{babel}
\usepackage{amsfonts}
\usepackage{amscd}
\usepackage{graphicx}
\usepackage{babel}
\usepackage{fancyhdr}
\usepackage[a4paper,left=2cm,right=2cm,top=2cm,bottom=2cm]{geometry}
\usepackage[french]{minitoc}
\usepackage{tabularx}
\usepackage{amssymb}
\usepackage{amsthm}
\usepackage{multirow}
\usepackage{multicol}
\usepackage{array}
\usepackage[all]{xy}
\usepackage[centertags]{amsmath}
\usepackage{latexsym}
\usepackage{amsfonts}
\usepackage{array}
\newtheorem{theorem}{\textbf{Theorem}}[section]

\theoremstyle{plain}

\theoremstyle{plain}

\theoremstyle{plain}

\theoremstyle{plain}

\theoremstyle{plain}

\theoremstyle{plain}

\theoremstyle{plain}

\theoremstyle{plain}
\newtheorem{Propri t }{\textbf{Propri t }}[section]
\theoremstyle{plain}

\theoremstyle{plain}

\theoremstyle{plain}

\begin{document}\begin{center}
\textbf{Solution of second-order hyperbolic quasilinear
systems with  spatio-characteristic initial data in weighted Sobolev-type spaces under finite differentiability assumptions on the data}
\end{center}
\begin{center}
 \textbf{Louokdom Tamto Paul Giscard$^{1}$,   Duplex Elvis Houpa
 Danga$^{2}$ and and Kouakep Tchaptchi  Yannick$^{3}$}
 \end{center}
 \begin{center}
 \begin{small}
$^{1}$Department of Mathematics, Faculty of Sciences, The University of
Ngaoundere, Dang, Ngaoundere, 454, ADAMAWA, CAMEROON.\\
$^{2}$Department of Mathematics, Faculty of Sciences, The University of
Ngaoundere, Dang, Ngaoundere, 454, ADAMAWA, CAMEROON.\\
$^{3}$Department of SFTI-EGCIM, The University of
Ngaoundere, Dang, Ngaoundere, 454, ADAMAOUA, CAMEROON.
\end{small}
\end{center}
\begin{center}
*Corresponding author(s). E-mail(s) :  louokdompaulgiscard@gmail.com; \\
Contributing authors: elvishoupa@gmail.com ; kouakep@aims-senegal.org.
\end{center}
\textbf{Abstract} \par The aim of this work is to establish an existence
and uniqueness solution  for  spatio-characteristic second-order quasilinear hyperbolic
problems in  Sobolev type spaces  with weights    to clarify and complete the
previous work done by  H. Muller Zum Hagen and H.J. Seifert, Gen.
Rel. and Gravit. 1977.
We  use this result in  P. G. Louokdom tamto, PhD thesis ongoing, 2026  to establish a semi-global existence and uniqueness result for second-order quasilinear Goursat problems where the coefficients of the second derivatives depend linearly on the unknown in weighted Sobolev-type spaces, which we will apply to the harmonic gauge vacuum Einstein equations.
 \begin{center}
\end{center}
\section{ Introduction}
\par
   However already knowing to solve locally in Sobolev space with  weight for second
   order hyper-quasilinear hyperbolic Goursat problem \cite{dt}, and even globally under small
    norms of the initial data values on a characteristic conoid, we envisage to establish
    a semi-global existence and uniqueness result \cite{lo}  in Sobolev spaces with weights
    for second order hyper-quasilinear hyperbolic Goursat problems where the coefficients
     of the derivatives depend linearly on the unknowns variable which we will apply to
     the Einstein vacuum in harmonic gauge. This is done by closely following
     the approach initiated by M. Dossa and S. Bah for the second order
     semi-linear Cauchy problems with initial data value on characteristic
     conoid \cite{db} and used by D. Houpa and M. Dossa for second order
     semi-linear Goursat problems \cite{dh}.\par~ The establishment of such a
      semi-global existence and uniqueness result in Sobolev space
       with weight passing through the local resolution of second
       order hyper-quasilinear spatio-characteristic problems whose existence comes from the foliation by spatial surfaces of the domain delimited by the two intersecting characteristic hypersurfaces which carry the initial data depending on the semi-global resolution of the characteristic problems as initiated by J. Tolen \cite{t} during the semi-global resolution  approach to second-order linear Goursat problems. We
       propose in this work to establish this local existence and uniqueness
       result for second order hyper-quasilinear hyperbolic spatio-characteristic problems. Similar
        results wherer established in \cite{do} for second order hyper-quasilinear hyperbolic Cauchy
         problems with initial data value on a characteristic conoid and successfully applied to the
          local resolution of Einstein vacuum equations in harmonic gauge and in  \cite{dt} when the
   initial data are from Goursat and applied to the local resolution of
  Einstein-Yang-Mills-Higgs equations in harmonic gauge and Lorentz gauge. We therefore
 use a method similar to that used in \cite{dt} to establish our result. This fixed-point
 method whose main tools are: the existence and uniqueness result for the second order
linear hyperbolic problems obtained, the Sobolev type inequalities
and the Dionne's type lemma.\par~ The remainder of this work is
organised as it follows. In section 2, we establish the
result of existence and uniqueness of the second order linear
hyperbolic spatio-characteristic problems in Sobolev type spaces
with weights $(\ref{fij})$, clarifying and complement previous work
of H. Muller Zum Hagen and H.J. Seifert \cite{ms}. We
establish a local existence and uniqueness result of the second
order hyperquasilinear hyperbolic spatio-characteristic problems
$(\ref{6999tt})$ in section 3. Finally, there is a short discussion
in section 4.

\section{ Local solution of the spatio-characteristic problem
second-order hyperbolic linear}
 Let
$L=A^{\lambda\mu}(x^{\alpha})\partial^{2}_{\lambda\mu}$ an operator $x^{0}$-hyperbolic differential of the second order in $
\mathcal{Y}$.\\Consider the linear problem
 with spatio-characteristic initial data next:
\begin{equation}\label{fij}
\left\{
                \begin{array}{ll}
E_{r}~:~A^{\lambda\mu
}(x^{\alpha})\partial^{2}_{\lambda\mu}u_{r}+B^{\lambda
z}_{r}(x^{\alpha})\partial_{\nu}u_{z}
+C^{z}_{r}(x^{\alpha})u_{z}=f_{r}(x^{\alpha})
 ~~~in ~~~ \mathcal{Y}  \hbox{} \\
 u_{r}=\varphi_{r} ~~on ~~S^{1}~,~~
u_{r}=h_{r},~~\partial_{0}u_{r}=g_{r}~~on ~~S^{2} .
                \end{array}
              \right.
\end{equation}where
$\alpha,\lambda,\mu,\nu=0,...,N$,~~ $r,z=1,...,n$.\\
We also have the following summary notation:
\begin{equation}\label{f}
\left\{
                \begin{array}{ll}
E~:~A^{\lambda\mu}D_{\lambda\mu}u+B.Du
+C.u=f ~~~in ~~~ \mathcal{Y}  \hbox{} \\
 u=\varphi~~on ~~S^{1}~,~~
u=h,~~Du=g~~on ~~S^{2} .
                \end{array}
              \right.
\end{equation}with
$u=(u_{r})$,~~ $D_{\lambda\mu}u=(\partial^{2}_{\lambda\mu}u_{r})$,~~ $Du=(\partial_{\nu}u_{r})$,~~
$B=(B^{\lambda
z}_{r})$,~~$C=(C^{z}_{r})$,~~$f=(f_{r})$,\\$\varphi=(\varphi_{r})$,~~$g=(g_{r})$,~~$h=(h_{r})$,~~$E=(E_{r})$,~~
$\lambda,\mu=0,...,N$.\\
 We assume that:\\
$\bullet$ $\mathcal{Y}$ is a compact part of $\mathbb{R}_{+}\times \mathbb{R}^{N}$ defined by: $\lbrace (x^{0} ,x^{i}) \in \mathbb{R}_{+}\times \mathbb{R}^{N}: -x^{1}\leq x^{0}, (x^{2} ,...,x^{N}) \in B \rbrace$ where $B$ is closed bounded subset of $\mathbb{R}^{N-1}$ whose boundary $\partial \mathcal{Y}$ is regular and includes:\\
$-$ the union of the two hypersurfaces $S=S^{1}\cup S^{2}$;\\
$-$ $\Omega=\partial \mathcal{Y} \setminus S$ is a temporal hypersurface of class $C^{\infty}$ piecewise;\\
$\bullet$
$L=A^{\lambda\mu}(x^{\alpha})\partial^{2}_{\lambda\mu}$ an operator $x^{0}$-hyperbolic differential of the second order in $
\mathcal{Y}$;\\
$\bullet$ $S^{1}$ is a characteristic hypersurface of
$\mathbb{R}_{+}\times \mathbb{R}^{N}$ for operator
$L=A^{\lambda\mu}(x^{\alpha})\partial^{2}_{\lambda\mu}$ and of
equation $x^{0}=-x^{1}$.\\$\bullet$ $S^{2}$ is a spatial
hypersurface of $\mathbb{R}_{+}\times\mathbb{R}_{+}\times
\mathbb{R}^{N-1}$ for the operator
$L=A^{\lambda\mu}(x^{\alpha})\partial^{2}_{\lambda\mu}$
and with equation $x^{0}=0$.\\
$\bullet$ For any function $v$ defined in a domain $\mathcal{Y}$, we set $[v]^{w}=v\mid_{S^{w}}$ \begin{center} $[v]^{w}(x^{1},x^{2},x^{3})=v((w-2)x^{1},x^{1},x^{2},x^{3})$, $(w=1,2)$. \end{center}
$\bullet$ $\Gamma=S^{1}\cap S^{2}$ and is spatial for $L$;\\
In, unlike the spatial problem associated with
$E_{r}$ equations, knowledge of the initial data
$u_{r}=\varphi_{r}$ on $S^{1}$ makes it possible to determine for any
solution $u=(u_{r})$ of the Goursat problem associated with the equations
$E_{r}$, the functions $[\partial_{0}u_{r}]^{1}$ in a unique way, and consequently to determine the restrictions to $S^{1}$
of all the first derivatives of the possible solution through
the relationship
$[\partial_{i}u_{r}]^{1}=\partial_{i}\varphi_{r}+q^{1}_{i}[\partial_{0}u_{r}]^{ 1}$. this is what justifies
 the birth of a problem of singularity of the initial data
on the secant $\Gamma$. To resolve this singularity problem, it is important to impose a compatibility hypothesis on the initial data. This justifies the following compatibility hypothesis:
 \\
$\bullet$ \textbf{Compatibility assumption:}
$\partial^{k}_{0}[u]^{1}= [\partial^{k}_{0}u]^{2}$ for all
 $k$ $\in$ $[0,m]\cap \mathbb{N}$ on $\Gamma$ where $m$ is the ordre of differentiability of any possible solution of the problem $(\ref{fij})$.\\
$\bullet$ $ \mathcal{Y}$ is causal: $\forall$ $M_{0}$ $\in$ $
\mathcal{Y}$, $(C_{M_{0}}^{-}) \subseteq \mathcal{Y}$ and $M_{0}$ is
 the unique singular point of $(C_{M_{0}}^{-})$; where $C_{M_{0}}^{-}$ is the characteristic half-conoid for $L$  coming from $M_{0}$ and directed towards the past;  $(C_{M_{0}}^{-})$ the of $C_{M_{0}}^{-}$ located between $M_{0}$ and $S$;\\

$\bullet$ $L$ is regularly hyperbolic of class $C^{1}$ on $\mathcal{Y}$ with hyperbolicity constant $h$ defined in the following sense:\\
-there is a positive constant $h_{1}$ such that:
\begin{equation}\label{tssti}
\gamma^{00}> h_{1}.
\end{equation}
- there is a positive constant $h_{2}$ such that:
\begin{equation}\label{tsst}
\gamma^{ij}X_{i}X_{j}> h_{2}\sum\limits_{i=1}^{N}(X_{i})^{2},
\forall~~(X_{i})\neq 0.
\end{equation}-there is a positive constant $h_{3}$ such that:
\begin{equation}\label{tssst}
\gamma^{\alpha\beta} X_{\alpha}X_{\beta}<
h_{3}\sum\limits_{\alpha=0}^{N}(X_{\alpha})^{2},
\forall~~(X_{\alpha})\neq 0.
\end{equation}

 with
\begin{equation}\label{tssst}
\gamma^{ij}=-A^{ij},~~\gamma^{00}=A^{00},~~\gamma^{0i}=0.
\end{equation}
We set: $T_{0}=\max\lbrace x^{0}, (x^{0},x^{i}) \in \mathcal{Y} \rbrace$,
 $h=\max\lbrace h_{1}^{-1}, h_{2}^{-1}, h_{3}\rbrace$.
We set: $T_{0}=\max\lbrace x^{0}, (x^{0},x^{i}) \in \mathcal{Y} \rbrace$,
 $h=\max\lbrace h_{1}^{-1}, h_{2}^{-1}, h_{3}\rbrace$.
$\forall$ $t$ $\in$ $]0,T_{0}]$, we define the following spaces:
 $\mathcal{Y}_{t}=\lbrace (x^{0},x^{i}) \in \mathcal{Y}: x^{0}\leq t\rbrace$,
  $G_{t}=\lbrace (x^{0},x^{i}) \in \mathcal{Y}: x^{0}=t\rbrace$,
   $S^{w}_{t}=S^{w}\cap \mathcal{Y}_{t}$, $\mathcal{D}_{t}^{w}=\lbrace (x^{ i})
\in \mathbb{R}^{N}; (x^{0},x^{i}) \in S^{w}_{t}\rbrace$,
$\sum_{t}^{1}=S^{1}\cap G_{t}$, $\Omega_{t}=\Omega\cap
\mathcal{Y}_{t}$.
\\
$\bullet$ For any vector function $v=(v_{r})$ defined on $\Gamma$ and for any $p$ $\in$ $\mathbb{N}$, we set:
 \\
$\parallel v\parallel_{
H^{p}(\Gamma)}=(\sum\limits_{r}\sum\limits_{\mid
\alpha\mid \leq p} \int_{\Gamma} \mid \partial^{\alpha} v_{r}
\mid ^{2} d\tilde{x}) )^{\frac{1}{2}}$;\\
 $\bullet$ For any vector function $v=(v_{r})$ defined on $S^{w}$ and for any $p$ $\in$ $\mathbb{N}$, we set:
 \\
$\parallel v\parallel_{
H^{p}(\sum_{t}^{1},S^{1})}=(\sum\limits_{r}\sum\limits_{\mid
\alpha\mid \leq p} \int_{\sum_{t}^{1}} \mid \partial^{\alpha} v_{r}
\mid ^{2} d\sigma (\sum_{t}^{1}) )^{\frac{1}{2}}$, where
$d\sigma (\sum_{t}^{1})$ is the measure induced on $\sum_{t}^{1}$ by $dx'=dx^{1} dx^{2}... dx^{N}$;
$\parallel v\parallel_{ H^{p}(S_{t}^{1})}=(\int^{t}_{0} \parallel
v\parallel^{2}_{ H^{p}(\sum_{\tau}^{1},S^{1})}d\tau)^{\frac{1}{2}}
$~~;~~
$\parallel v\parallel_{ F^{p}(S_{t}^{2})}=(\int^{t}_{0} \parallel v\parallel^{2}_{ H^{p }(S_{\tau}^{2})}d\tau)^{\frac{1}{2}} $;
$\parallel v\parallel_{
H^{p}(S_{t}^{2})}=(\sum\limits_{r}\sum\limits_{\mid \alpha\mid \leq
p} \int_{ S_{t}^{2}} \mid \partial^{\alpha}
 v_{r} \mid ^{2} dx' )^{\frac{1}{2}}$;
$\parallel v\parallel_{ E^{p}(S_{t}^{1})}=ess\sup \limits_{ \tau \in
]0, t]}\parallel v\parallel_{ H^{p}(\sum_{\tau}^{1},S^{1})} $~;
$\parallel v\parallel_{ E^{p}(S_{t}^{2})}=ess\sup \limits_{ \tau \in ]0, t]}\parallel v\parallel_{ H^{p }(S^{2}_{\tau})} $;
if the right-hand sides exist and the derivatives being taken in the sense of the distributions.\\
$\bullet$ For any vector function $v=(v_{r})$ defined on $\mathcal{Y}_{t}$, we set: $\forall$ $p$ $\in$ $\mathbb{ N}$,\\
$\parallel v\parallel_{
H^{p}(G_{\tau},\mathcal{Y})}=(\sum\limits_{r}\sum\limits_{\mid
\alpha\mid \leq p} \int_{G_{\tau}} \mid  D^{\alpha} v_{r} \mid ^{2}
dx' )^{\frac{1}{2}}$; $\parallel v\parallel_{
K^{p}(\mathcal{Y}_{t})}=(\int_{0}^{t}\tau^{-1} \parallel
v\parallel^{2}_{ H^{p}(G_{\tau},\mathcal{Y})} d\tau )^{\frac{1}{2}}$;\\
$\parallel v\parallel_{ E^{p}(\mathcal{Y}_{t})}= ess\sup \limits_{
\tau \in ]0, t]} \tau^{-\frac{1}{2}} \parallel v\parallel_{
H^{p}(G_{\tau},\mathcal{Y})} $; $\parallel v\parallel_{
\mathcal{K}^{p}(\Gamma,S^{1},\mathcal{Y})}=
(\sum\limits_{k=0}^{p-1}  \parallel [\partial^{k}_{0}v]^{1} \parallel^{2}_{ H^{2(p-k)-1}(\Gamma,S^{1})}   )^{\frac{1}{2}}$;\\
 $\parallel v\parallel_{ \mathcal{K}^{p}_{1}(\Gamma,S^{1},\mathcal{Y})}=
(\sum\limits_{k=1}^{p-1}  \parallel [\partial^{k}_{0}v]^{1}
\parallel^{2}_{ H^{2(p-k)-1}(\Gamma,S^{1})}   )^{\frac{1}{2}}$;
$\parallel v\parallel_{
\mathcal{K}^{p}(\Gamma,S^{2},\mathcal{Y})}=(\sum\limits_{k=0}^{p-1}
\parallel [\partial^{k+1}_{0}v]^{2} \parallel^{2}_{
F^{p-k-2}(\Gamma,S^{2})})^{\frac{1}{2}}$;\\
$\parallel v\parallel_{
\mathcal{K}^{p}_{1}(\Gamma,S^{2},\mathcal{Y})}=(\sum\limits_{k=1}^{p-1}
 \parallel [\partial^{k+1}_{0}v]^{2} \parallel^{2}_{ F^{p-k-2}(\Gamma,S^{2})})^{\frac{1}{2}}$;
$\parallel v\parallel_{
\mathcal{K}^{p}_{1}(S_{t}^{1},\mathcal{Y})}=(\sum\limits_{k=1}^{p-1}
\parallel [\partial^{k}_{0}v]^{1} \parallel^{2}_{
H^{2(p-k)-1}(S_{t}^{1})}   )^{\frac{1}{2}}$ ; \\$\parallel
v\parallel_{ \mathcal{K}^{p}(S_{t}^{1},\mathcal{Y})}=
(\sum\limits_{k=0}^{p-1}  \parallel [\partial^{k}_{0}v]^{1}
\parallel^{2}_{ H^{2(p-k)-1}(S_{t}^{1})}   )^{\frac{1}{2}}$ ;
$\parallel v\parallel_{ \mathcal{K}^{p}(S_{t}^{2},\mathcal{Y})}=
(\sum\limits_{k=0}^{p-1}  \parallel [\partial^{k+1}_{0}v]^{2}
\parallel^{2}_{ F^{p-k-1}(S_{t}^{2})}   )^{\frac{1}{2}}$ ;\\
$\parallel v\parallel_{ \mathcal{K}^{p}_{1}(S_{t}^{2},\mathcal{Y})}=
(\sum\limits_{k=1}^{p-1}  \parallel [\partial^{k+1}_{0}v]^{2}
\parallel^{2}_{ F^{p-k-1}(S_{t}^{2})}   )^{\frac{1}{2}}$ ; $\parallel
v\parallel_{ \mathcal{E}^{p}_{1}(S_{t}^{1},\mathcal{Y})}=(
 \sum\limits_{k=1}^{p-1}  \parallel [\partial^{k}_{0}v]^{1} \parallel^{2}_{
 E^{2(p-k)-1}(S_{t}^{1})})^{\frac{1}{2}}$;\\
$\parallel v\parallel_{ \mathcal{E}^{p}_{1}(S_{t}^{2},\mathcal{Y})}=
  (\sum\limits_{k=1}^{p-1}  \parallel [\partial^{k+1}_{0}v]^{2}\parallel^{2}_{ E^{p-k-1}(S_{t}^{2})}   )^{\frac{1}{2}}$;
$\parallel v\parallel_{
\mathcal{E}^{p}(S_{t}^{1},\mathcal{Y})}=(\sum\limits_{k=0}^{p-1}
\parallel [\partial^{k}_{0}v]^{1} \parallel^{2}_{
E^{2(p-k)-1}(S_{t}^{1})})^{\frac{1}{2}}$;\\ $\parallel v\parallel_{
\mathcal{E}^{p}(S_{t}^{2},\mathcal{Y})}=(
 \sum\limits_{k=0}^{p-1}  \parallel[ \partial^{k+1}_{0}v]^{2}\parallel_{ E^{p-k-1}(S_{t}^{2})}   )^{\frac{1}{2}}
 $;
$\parallel v\parallel_{ \mathcal{K}^{p}_{1}(\mathcal{Y}_{t})}=
\parallel v \parallel_{ K^{p}(\mathcal{Y}_{t})}+
  \sum\limits_{w=1}^{2}\parallel v \parallel_{ \mathcal{K}^{p}_{1}(S_{t}^{w},
\mathcal{Y})}   $;\\
$\parallel v\parallel_{ \mathcal{K}^{p}(\mathcal{Y}_{t})}=
\parallel v \parallel_{ K^{p}(\mathcal{Y}_{t})}+
  \sum\limits_{w=1}^{2}\parallel v \parallel_{ \mathcal{K}^{p}(S_{t}^{w},
\mathcal{Y})}  $; $\parallel v\parallel_{
\mathcal{E}^{p}_{1}(\mathcal{Y}_{t})}=\parallel v \parallel_{
E^{p}(\mathcal{Y}_{t})}+
  \sum\limits_{w=1}^{2} \parallel v \parallel_{ \mathcal{E}^{p}_{1}(S_{t}^{w}, \mathcal{Y})}   $
  ;
$\parallel v\parallel_{ \mathcal{E}^{p}(\mathcal{Y}_{t})}= \parallel
v \parallel_{ E^{p}(\mathcal{Y}_{t})}+
  \sum\limits_{w=1}^{2} \parallel v \parallel_{ \mathcal{E}^{p}(S_{t}^{w}, \mathcal{Y})} $ ;
if the right-hand sides exist and the derivatives being taken in the sense of the distributions.\\we establish the following
result of existence and uniqueness of the second order linear
hyperbolic spatio-characteristic problems in Sobolev type spaces
with weights $(\ref{fij})$, clarifying and complement previous work
of H. Muller Zum Hagen and H.J. Seifert \cite{ms}.
 \begin{theorem} \label{40ssytt}We assume that\\
$\bullet$ $s$ is a natural number $> \frac{N}{2}+1$, $N\geq 2$;
  the functions $A^{\lambda \mu}$ verify the hypotheses $(\ref{tssti}-\ref{tssst})$;
 the functions $A^{\lambda \mu}$, $B$, $C$ are elements of $\mathcal{K}^{s}(Y_{T})$;
 the function $f$ is an element of $\mathcal{K}^{s-1}(Y_{T})$;\\
$\bullet$ the functions $\varphi$ are elements of $H^{2s-1}(S_{T}^{1})$ such that $\varphi\mid_{\Gamma}$ $\in$ $ F^{2s-2}(\Gamma,S^{2})$;
 the functions $h$ and $g$ are respectively the elements of $F^{s}(S_{T}^{2})$ and $F^{s-1}(S_{T}^ {2})$;\\
 $\Gamma=S^{1}\cap S^{2}$ and is spatial for $L$;
 $\partial^{k}_{0}[u]^{1}= [\partial^{k}_{0}u]^{2}$ for all $k$ $\in$ $[0,s]\cap \mathbb{N}$ on $\Gamma=S^{1}\cap S^{2}$.\\
Then:~~1) $(\ref{fij})$ admits a unique solution $u$ $\in$
$\mathcal{E}_{1}^{s}(Y_{T})$ and this solution verifies: $\forall$
$t$ $\in$ $]0,T]$,

\begin{equation}\label{4101ksiiisiiiu}
                \parallel u\parallel_{\mathcal{E}_{1}^{s}(Y_{t})} \leq c(h,\tilde{\tilde{\gamma}}_{s}^{Y_{t}},\tilde{\tilde{\gamma}}_{s,}^{\Gamma},t) \left\{
                \begin{array}{ll}   \parallel \varphi \parallel_{H^{2s-1}(S_{t}^{1})} +  \parallel \varphi \parallel_{F^{2s-2}(\Gamma,S^{2})}
                +\parallel h \parallel_{F^{s}(S_{t}^{2})} & \hbox{} \\+\parallel g \parallel_{F^{s-1}(S_{t}^{2})}+
                 t^{\frac{1}{2}} \parallel f \parallel_{\mathcal{K}^{s-1}(Y_{t})}+\sum\limits^{2}_{w=1}\parallel f \parallel_{\mathcal{K}^{s-1}(\Gamma,S^{w}, Y)} \end{array}
              \right\}.
\end{equation} with:
 \begin{equation}\label{bcxr}
    \left\{
                \begin{array}{ll} \tilde{\tilde{\gamma}}_{s}^{Y_{t}}=[\parallel A^{\lambda\mu}\parallel^{2}_{\mathcal{K}^{s}(Y_{t})}+\parallel B\parallel^{2}_{\mathcal{K}^{s}(Y_{t})}+\parallel C\parallel^{2}_{\mathcal{K}^{s}(Y_{t})}]^{\frac{1}{2}},  & \hbox{} \\
              \tilde{\tilde{\gamma}}_{s}^{\Gamma}  =\sum\limits_{w=1}^{2}[\parallel [A^{\lambda\mu}]^{w}\parallel^{2}_{\mathcal{K}^{s}(\Gamma, S^{w})}+\parallel [B]^{w}\parallel^{2}_{\mathcal{K}^{s}(\Gamma, S^{w})}+\parallel [C]^{w}\parallel^{2}_{\mathcal{K}^{s}(\Gamma, S^{w})}]^{\frac{1}{2}}
                 \end{array}
              \right.
\end{equation} 2) If in addition the  functions $\varphi$
are elements of $E^{2s-1}(S_{T}^{1})$ and such that
$\varphi\mid_{\Gamma}$ $\in$ $E^{2s-2}(\Gamma,S^{2})$, $h$ $\in$
$E^{s}(S_{T}^{2})$, $g$ $\in$ $E^{s-1}(S_{T}^{2})$ then
$(\ref{fij})$   admits a unique solution $u$ $\in$
$\hat{\mathcal{E}}^{s}(Y_{T})$ and this solution verifies:
 $\forall$ $t$ $\in$ $]0,T]$,
\begin{equation}\label{4101ksiiisiiiut}
                \parallel u\parallel_{\mathcal{E}^{s}(Y_{t})} \leq c(h,\tilde{\tilde{\gamma}}_{s}^{Y_{t}},\tilde{\tilde{\gamma}}_{s,}^{\Gamma},t) \left\{
                \begin{array}{ll}   \parallel \varphi \parallel_{E^{2s-1}(S_{t}^{1})} +
                 \parallel \varphi \parallel_{E^{2s-2}(\Gamma,S^{2})} +\parallel h \parallel_{E^{s}(S_{t}^{2})} & \hbox{} \\
                 +\parallel g \parallel_{E^{s-1}(S_{t}^{2})}+ t^{\frac{1}{2}} \parallel f \parallel_{\mathcal{K}^{s-1}(Y_{t})}
                 +\sum\limits^{2}_{w=1}\parallel f \parallel_{\mathcal{K}^{s-1}(\Gamma,S^{w}, Y)}   \end{array}
              \right\}.
\end{equation} \end{theorem}
\textbf{Proof:} It is done in a similar way to that of theorem 6.2.2.4 of \cite{hou}.
\\we establish the following substitution lemmas which help us later.

\begin{theorem} \label{33ss}Let $u=(u_{m})_{1\leq m\leq l}$ be an
element of $\mathcal{E}^{p}(Y_{t}) $
 ; $Z$ an open of $\mathbb{R}^{l}$, $f$ an element of $C^{2r-1}_{b}(Y_{t}\times Z)$ .\\We assume that $1\leq r \leq p$, $\frac{N}{2}< p$; $u:$ $x\mapsto$ $u(x)=(u_{1}(x),...,u_{l}(x))$ is a mapping from $Y_{t}$ into $Z $. We note $\tilde{f}_{u}$: $x\mapsto$ $\tilde{f}_{u}(x)=f(x,u(x))$.\\Then $\tilde {f}_{u}$ satisfies the following inequality:
\begin{equation}\label{34ss}
\parallel \tilde{f}_{u} \parallel_{\mathcal{E}^{r}(Y_{t})}\leq c_{t}(r,p)\parallel \tilde{f}_{ u} \parallel_{ C^{2r-1}_{b}(Y_{t}\times Z)}[1
+\parallel u \parallel_{\mathcal{E}^{p}(Y_{t}) }]^{2r-1}.
\end{equation}
As $\mathcal{E}^{r}(Y_{t})$ injects contin$\hat{u}$ment into
$\mathcal{K}^{r}(Y_{t})$ , we also have:
\begin{equation}\label{35ss}
\parallel \tilde{f}_{u} \parallel_{\mathcal{K}^{r}(Y_{t})}\leq c_{t}(r,p)\parallel \tilde{f}_{ u}
 \parallel_{ C^{2r-1}_{b}(Y_{t}\times Z)}[1+\parallel u \parallel_{\mathcal{E}^{p}(Y_{t}) }]^{2r-1}.
\end{equation}
\end{theorem}
\textbf{Proof:} It is done in a similar way to that of theorem 6.2.2.6 of \cite{hou}.
\begin{theorem}\label{36ss} Let $u=(u_{m})_{1\leq m\leq l}$,
$v=(v_{m})_{1\leq m\leq l}$ be elements
 of $\mathcal{E}^{p-1}(Y_{t})$; $Z$ an open of $\mathbb{R}^{l}$, $f$ an element of $C^{2r-3}_{b}(Y_{t}\times Z)$
  such as $\forall$ $i$ $\in$ $[1,l]\cap \mathbb{N}$, $D_{u_{i}}f$ $\in$ $C^{2r-3} _{b}(Y_{t}\times Z)$
  .\\We suppose that $2\leq r \leq p$, $\frac{N}{2}+1< p$ ; $u:$ $x\mapsto$ $u(x)=(u_{1}(x),...,u_{l}(x))$,
   $v:$ $x\mapsto$ $v(x)=(v_{1}(x),...,v_{l}(x))$ are maps of $Y_{t}$ in $Z $. \\We denote $\tilde{f}_{u}$: $x\mapsto$ $\tilde{f}_{u}(x)=f(x,u(x))$.\\Then $ \tilde{f}_{u}$ satisfies the following inequality:\\
a) Then we have:
\begin{equation}\label{37ss}
\parallel \tilde{f}_{u}-\tilde{f}_{u+v} \parallel_{\mathcal{E}^{r-1}(Y_{t})}\leq c_{t} (r,p) \left\{
                \begin{array}{ll} \max\limits_{1\leq i\leq l}\parallel D_{u_{i}}f \parallel_{ C^{2r-3}_{b}(Y_{t} \times Z)}\parallel v \parallel_{\mathcal{E}^{p-1}(Y_{t})}
      \hbox{} \\ \times[1+\parallel u
\parallel_{\mathcal{E}^{p-1}(Y_{t})}+\parallel v
\parallel_{\mathcal{E}^{p -1}(Y_{t})}]^{2r-3}.
       \end{array}
              \right.
\end{equation}
b) If in addition $w$ $\in$ $\mathcal{E}^{p-1}(Y_{t})$ then we have:
\begin{equation}\label{38ss}
\parallel (\tilde{f}_{u}-\tilde{f}_{u+v})w \parallel_{\mathcal{E}^{r-1}(Y_{t})}\leq c_ {t}(r,p) \left\{
                \begin{array}{ll} \max\limits_{1\leq i\leq l}\parallel D_{u_{i}}f \parallel_{ C^{2r-3}_{b}(Y_{t} \times Z)}\parallel v \parallel_{\mathcal{E}^{p-1}(Y_{t})}\times \parallel w \parallel_{\mathcal{E}^{p-1}(Y_ {t})}

     \hbox{} \\ \times[1+\parallel u \parallel_{\mathcal{E}^{p-1}(Y_{t})}+\parallel v \parallel_{\mathcal{E}^{p
  -1}(Y_{t})}]^{2r-3}.
       \end{array}
              \right.
\end{equation}
c) If we further suppose that $r+1\leq p$, $w$ $\in$
$\mathcal{E}^{p-2}(Y_{t})$ ,$f$ an element of
$C^{2r-1}_{b}(Y_{t}\times Z)$ such that $\forall$ $i$ $\in$
$[1,l]\cap \mathbb{N}$, $D_{u_{i}}f$ $\in$ $C^{2r-1}_{b}(Y_{t}\times
Z)$ then we have:
\begin{equation}\label{39ss}
\parallel (\tilde{f}_{u}-\tilde{f}_{u+v})w \parallel_{\mathcal{E}^{r-1}(Y_{t})}\leq c_ {t}(r,p) \left\{
                \begin{array}{ll} \max\limits_{1\leq i\leq l}\parallel D_{u_{i}}f \parallel_{ C^{2r-1}_{b}(Y_{t} \times Z)}\parallel v \parallel_{\mathcal{E}^{p-1}(Y_{t})}\times \parallel w \parallel_{\mathcal{E}^{p-2}(Y_ {t})}

     \hbox{} \\ \times[1+\parallel u \parallel_{\mathcal{E}^{p-1}(Y_{t})}+\parallel v \parallel_{\mathcal{E}^{p
  -1}(Y_{t})}]^{2r-1}.
       \end{array}
              \right.
\end{equation}
d) Under the same  respective assumptions and under the subsidiary
hypotheses that $v$ $\in$ $\mathcal{K}^{p-1}(Y_{t})$ on has the
analogues of the inequalities $(\ref{37ss}-\ref{39ss})$ to condition
to replace on the left $\mathcal{E}^{p-1}(Y_{t})$ by
$\mathcal{K}^{p-1}(Y_{t})$ and replace outside brackets in the right
member $\parallel v
\parallel_{\mathcal{E}^{p-1}(Y_{t})}$ by $\parallel v
\parallel_{\mathcal{K}^{p-1}(Y_{t})}$.
\end{theorem}
\textbf{Proof:} It is done in a similar way to that of theorem 6.2.2.6 of \cite{hou}.
\section{ Local solution of the spatio-characteristic problem
second-order hyperbolic hyper-quasilinear}  Let
$\bar{L}=A^{\lambda\mu}(x^{\alpha},\bar{u}_{z}(x^{\alpha}))\partial^{2}_{\lambda\mu}$
an operator $x^{0}$-hyperbolic differential of the second order in $
\mathcal{Y}$.
 Consider the problem
hyper-quasilinear with spatio-characteristic initial data next:
\begin{equation}\label{6999tt}
\left\{
                \begin{array}{ll}
(G_{r})~:~A^{\lambda\mu
}(x^{\alpha},u_{z}(x^{\alpha}))\partial^{2}_{\lambda\mu}u_{r}+f_{r}(x^{\alpha},u_{z},\partial_{\nu}u_{z})=0
~~~in ~~~ \mathcal{Y}, \hbox{} \\ u_{r}=\bar{u}_{r}=\varphi_{r} ~~on
~~S^{1}~,~ u_{r}=h_{r},~~\partial_{0}u_{r}=g_{r}~~on ~~S^{2}\end{array}
              \right.
\end{equation}where
$\alpha,\lambda,\mu,\nu=0,...,N$,~~ $r,z=1,...,n$.\\
 We assume that:\\
$\bullet$ $\bar{u}$ is an application defined on $\mathcal{Y}$,
continues on secant $\Gamma$ such that $\bar{u}(\Gamma)\subset W$
and
$D\bar{u}(S^{w})\subset Z$ where $W$ is an open set of $\mathbb{R}^{n}$ and $Z$ an open set of $\mathbb{R} ^{n(N+1)}$;\\
$\bullet$ $S^{1}$ is a characteristic hypersurface of
$\mathbb{R}_{+}\times \mathbb{R}^{N}$ for operator
$\bar{L}=A^{\lambda\mu}(x^{\alpha},\bar{u}_{z}(x^{\alpha}))\partial^{2}_{\lambda\mu}$
and with equation $x^{0}=-x^{1}$.\\
$\bullet$ $S^{2}$ is a spatial hypersurface of
$\mathbb{R}_{+}\times\mathbb{R}_{+}\times \mathbb{R}^{N-1}$ for the
operator
$\bar{L}=A^{\lambda\mu}(x^{\alpha},\bar{u}_{z}(x^{\alpha}))\partial^{2}_{\lambda
\mu}$
and with equation $x^{0}=0$.\\
$\bullet$ $\Gamma=S^{1}\cap S^{2}$ and is spatial for $\bar{L}$.
\\ We
establish  the following result of  existence and uniqueness  of the second
order hyperquasilinear hyperbolic spatio-characteristic problems.
\begin{theorem} Let $s$ $\in$ $\mathbb{N}$, $s\geq6$, $T$
$\in$ $\mathbb{R}^{*}_{+}$. If the $A^{\lambda\mu}$ are functions of
class $C^{\infty}$ on
 $\mathcal{Y} \times W$, the $f_{r}$ are functions of class $C^{\infty}$ on $\mathcal{Y} \times W\times Z$ and the data
 initials $\varphi=(\varphi_{r})$, $h=(h_{r})$, $g=(g_{r})$ are such that:\\
(i) $\varphi$ $\in$ $\hat{E}^{2s-1}(S_{T}^{1})$ such that $\varphi\mid_{\Gamma}$ $\in$ $\hat{E}^{2s-2}(\Gamma,S^{2})$;
(ii) $h$ $\in$ $\hat{E}^{s}(S^{2})$, $g$ $\in$ $\hat{E}^{s-1}(S ^{2})$;\\
(iii) $\partial^{k}_{0}[u]^{1}= [\partial^{k}_{0}u]^{2}$ for all $k$ $\in$ $ [0,s]\cap \mathbb{N}$ on $\Gamma=S^{1}\cap S^{2}$.\\
So:\\(a) $\exists$ $T_{0}$ $\in$ $]0,T]$, such as in the domain $\mathcal{Y}_{T_{0}}$, the spatio-characteristic hyper-quasilinear problem $(\ref{6999tt})$ admits a unique solution $u=(u_{r})$ $\in$ $\hat{\mathcal{E}}^{s}(\mathcal{Y} _{T_{0}})$.\\
(b) If moreover the initial data $\varphi=(\varphi_{r})$ are
restrictions to $S^{1}$ of functions of class $C^{\infty}$ on
$\mathbb{R}^{2}\times B$, and the initial data $h=(h_{r})$,
$g=(g_{r})$ are restrictions to $S^{2}$ of class functions
$C^{\infty}$ on $\mathbb{R}^{2}\times B$, then $u=(u_{r})$ $\in$
$C^{\infty}(\mathcal{Y}_{T_{0}})$.\end{theorem} \textbf{Proof.} The
proof is done in four steps.
 We consider the application $F$:
$\hat{\mathcal{E}}^{s}(\mathcal{Y}_{T})\rightarrow
\hat{\mathcal{E}}^{s}(\mathcal{Y}_{T})$; $V=(V_{l})\mapsto (u_{r}) $
where $(u_{r})$ is solution of the  following spatio-characteristic linear problem
:
\begin{equation}\label{7000tt}
A^{\lambda\mu }(x^{\alpha},V_{l})\partial^{2}_{\lambda\mu}u_{r}+f_{r}(x^{\alpha},V_{l},\partial_{\nu}V_{l})=0 ~~in ~~ \mathcal{Y}_{T,}
~~ u_{r}=\bar{u}_{r}=\varphi_{r} ~~on ~~S^{1}~~  u_{r}=h_{r},~~\partial_{0}u_{r}=g_{r}~~on ~~S^{2}.
\end{equation}
At the first step, to ensure that this ball is non-empty, it is important to construct on $\mathcal{Y}_{T}$ a function $u^{0}$ such that $u^{0}$ $\in$ $\hat{\mathcal{E}}^{s}(\mathcal{Y}_{T})$ and $u^{0}=\bar{u}^{0}=\varphi^{w} $ on $S^{w}$. We construct an
element $u^{0}$ $\in$ $\hat{\mathcal{E}}^{s}(\mathcal{Y}_{T})$
solution of the  following  spatio-characteristic  linear problem:

\begin{equation}\label{715t}
\square_{3} u_{r}^{0}=0 ~~~in ~~~ \mathcal{Y}_{T}~~;~~
 u_{r}^{0}=\bar{u}_{r}^{0}=\varphi^{w}_{r} ~~on ~~S^{1}_{T}; ~~ u_{r}^{0}=h_{r},~~\partial_{0}u_{r}^{0}=g_{r}~~on ~~S^{2}_{T}.
\end{equation}It follows from theorem 8.1 of \cite{ms} that the previous spatio-characteristic linear admits a unique solution  $u^{0}=(u^{0}_{r})$  $\in$ $\hat{\mathcal{E}}^{s}(\mathcal{Y}_{T})$.At the second step, using Sobolev inequalities and substitution lemmas, we show that the map $F$ is well defined. Suppose that $A^{\lambda\mu}$ and $f_{r}$ are functions of class $C^{\infty}$ of their arguments. Since
  $\Gamma$ is a compact subset of $S^{w}$, it follows that $\bar{u}(\Gamma)$ is a subset
   compact of $W$ as an image of a compact part by a continuous application. As $W$ is a subset
    open of $\mathbb{R}^{n}$, there exists a positive real number $c$ such that: $\bar{V}_{c}=\bar{u}(\Gamma)+\bar{ B}_{c}\subset W$
     with $B_{c}=\lbrace u \in \mathbb{R}^{n}: \mid u\mid< c\rbrace$ and $\bar{B}_{c}=\lbrace u \in \mathbb{R}^{n}: \mid u\mid\leq c\rbrace$ where
      $\mid u\mid$ represents the Euclidean norm of $u$ in $ \mathbb{R}^{n}$. Likewise, since $\Gamma$ is a compact subset
      contained in $S^{w}$, it follows that $D\bar{u}(\Gamma)$ is a compact subset contained in $Z$ and as $Z$ is an open subset of $\mathbb{R}^{4n}$, there exists a positive real number $c'$ such that: $\bar{V}_{c'}=D\bar{u}(\Gamma)+\bar{B }_{c'}\subset Z$ with $B_{c'}=\lbrace Du \in \mathbb{R}^{4n}: \mid Du\mid< c'\rbrace$ and $\bar{B }_{c'}=\lbrace Du \in \mathbb{R}^{4n}: \mid Du\mid\leq c'\rbrace$ where $\mid Du\mid$ represents the Euclidean norm of $Du$ in $ \mathbb{R}^{4n}$. In order to show that the map $F$ is well defined, we must define the hyperbolicity constant $h$ which will allow us to use the theorem $\ref{40ssytt}$. Let $\gamma^{ij}(x^{\alpha},\bar{u}_{s})$ be the positive definite metric associated with the hyperbolic operator $\bar{L}$ defined as follows: \begin {center}
 $\gamma^{ij}(x^{\alpha},\bar{u}_{s})=-A^{ij}(x^{\alpha},\bar{u}_{s}) $, $\gamma^{00}(x^{\alpha},\bar{u}_{s})=A^{00}(x^{\alpha},\bar{u}_{s} )$ and
  $\gamma^{0i}(x^{\alpha},\bar{u}_{s})=0$. \end{center} We set $S^{3}$ to be the unit sphere of $\mathbb{R}^{4}$ and we define the constants $h_{1}$, $h_{2}$ and $h_ {3}$ as follows: \begin{center}
 $\gamma^{00}(x^{\alpha},\bar{u}_{s})> h_{1}$,

 \end{center}
\begin{center}
 $h_{2}=\inf\limits_{(x^{\alpha}, \bar{u}_{s},X_{i}) \in \mathcal{Y}_{T}\times \bar{ V}_{c}\times S^{3}}\gamma^{ij}(x^{\alpha},\bar{u}_{s})X_{i}X_{j}$,

 \end{center} and
  \begin{center}
 $h_{3}=\sup\limits_{(x^{\alpha}, \bar{u}_{s},X_{i}) \in \mathcal{Y}_{T}\times \bar{ V}_{c}\times S^{3}}\gamma^{ij}(x^{\alpha},\bar{u}_{s})X_{i}X_{j}$
 \end{center}
  Since $ \mathcal{Y}_{T}\times \bar{V}_{c}$ and $S^{3}$ are compacts then $h_{1}$, $h_{2}$ and $ h_{3}$ are positive constants.\\
Taking $h=\max\lbrace h_{1}^{-1}, h_{2}^{-1}, h_{3}\rbrace$
 it follows that $h$ is a hyperbolicity constant for the hyperbolic operator $A^{\lambda\mu}(x^\alpha,V_{l}(x^\alpha))\partial^{2}_ {\lambda\mu}$ this
 for any continuous function $V_{l}$: $Y_{T}\rightarrow \bar{V}_{c}$. The operator $A^{\lambda\mu}(x^\alpha,V_{l}(x^\alpha))\partial^{2}_{\lambda\mu}$ being regularly hyperbolic with constant  hyperbolicity $h$ by virtue of theorem $\ref{40ssytt}$, it suffices to show that for $V$ $\in$ $\hat{\mathcal{E}}^{s}(\mathcal{Y}_ {T})$ , the applications $\tilde{A}_{V}$: $(x^\alpha)$ $\mapsto$ $A^{\lambda\mu }(x^{\alpha},V_ {l})$ and $\tilde{f}_{V}$: $(x^\alpha)$ $\mapsto$ $f_{r}(x^{\alpha},V_{l},\partial_ {\nu}V_{l})$ belong respectively to the spaces $\mathcal{K}^{s}(\mathcal{Y}_{T})$ and $\mathcal{K}^{s-1}( \mathcal{Y}_{T})$.\\
Since $s>\frac{3}{2}+1$, we have $\mathcal{E}^{s}(\mathcal{Y}_{T})
\hookrightarrow C^{1}(\mathcal{Y}_{T})$ .
\begin{equation}\label{701h}
\left\{
                \begin{array}{ll}
 \parallel V\parallel_{C^{0}(\mathcal{Y}_{T})}\leq \epsilon \parallel V\parallel_{\mathcal{E}^{s}(\mathcal{Y}_{ T})}, \hbox{} \\
 \parallel DV\parallel_{C^{0}(\mathcal{Y}_{T})}\leq \epsilon \parallel V\parallel_{\mathcal{E}^{s}(\mathcal{Y}_{ T})} .
                \end{array}
              \right.
\end{equation}
by setting $r= \frac{1 }{\epsilon} \min( c,c')>0$ we have for
$\parallel V \parallel_{\mathcal{E}^{s}(\mathcal{Y}_{t})}\leq r$, $(V(x^{\alpha}),DV(x^{ \alpha})) \in V_{c}\times V_{c'}$,~$\forall$~ $(x^{\alpha})~\in~\mathcal{Y}_{T}$.
 \\
Considering the inequality $(\ref{34ss})$ in which we take respectively: $1\leq r=p=s-1$, we have:\begin{equation}\label{702}
\parallel \tilde{f}_{V}\parallel_{\mathcal{K}^{s-1}(\mathcal{Y}_{T})}\leq c^{1}_{T}(s )\parallel f\parallel_{C^{2s-3}_{b}(\mathcal{Y}_{T}\times \bar{V}_{c}\times \bar{V}_{c' })}\times[1
+\parallel V\parallel_{\mathcal{E}^{s}(\mathcal{Y}_{T})}]^{2s-3}
\end{equation} and
$1\leq r=p=s$, we have:\begin{equation}\label{723}
\parallel \tilde{A}_{V}\parallel_{\mathcal{K}^{s}(\mathcal{Y}_{T})}\leq c^{1}_{T}(s)\parallel A^{\lambda\mu} \parallel_{C^{2s-1}_{b}(\mathcal{Y}_{T}\times \bar{V}_{c})}\times [1
+\parallel V\parallel_{\mathcal{E}^{s}(\mathcal{Y}_{T})}]^{2s-1}.
\end{equation}
We then deduce respectively from the relations $(\ref{702})$ and $(\ref{723})$ that as soon as $V$ $\in$ $\hat{\mathcal{E}}^{s}( \mathcal{Y}_{T})$ we have $\tilde{f}_{V}$ $\in$ $\mathcal{K}^{s-1}(\mathcal{Y}_{T} )$ and $\tilde{A}_{V}$ $\in$ $\mathcal{K}^{s}(\mathcal{Y}_{T})$.
\\Hence $F$ is well defined as the application of $\hat{\mathcal{E}}^{s}(\mathcal{Y}_{T})$ into himself.\\Since $u=(u_{r})$ is solution to the linear problem $(\ref{7000tt})$, we deduce from inegality $(\ref{4101ksiiisiiiut})$ of  theorem $\ref{40ssytt}$  following inequality: $\forall$ $t$ $\in$ $]0,T]$,

\begin{equation}\label{703}
\parallel u \parallel_{\mathcal{E}^{s}(\mathcal{Y}_{t})}\leq c [ \parallel \varphi \parallel_{E^{2s-1}(S_{t} ^{1})} +
 \parallel \varphi \parallel_{E^{2s-2}(\Gamma,S^{2})} +\parallel h \parallel_{E^{s}(S^{2}_{t})}
 +\parallel g \parallel_{E^{s-1}(S^{2}_{t})}+t^{\frac{1}{2}}\parallel \tilde{f}_{V} \parallel_{\mathcal{K}^{s-1}(\mathcal{Y}_{t})}].
\end{equation}
Hence by combining the relations $(\ref{702})$ and $(\ref{703})$ we have:

\begin{equation}\label{704}\begin{aligned}
\parallel u \parallel_{\mathcal{E}^{s}(\mathcal{Y}_{t})}&\leq & c_{1} [ \parallel \varphi \parallel_{E^{2s-1} (S_{t}^{1})}
+ \parallel \varphi \parallel_{E^{2s-2}(\Gamma,S^{2})} +\parallel h
\parallel_{E^{s}(S^{2}_{t})} +\parallel g
\parallel_{E^{s-1}(S^{2}_{t})}\\ & &+t^{\frac{1}{2}}\parallel
f\parallel_{C^{2s-3}_{b}(\mathcal{Y}_{T}\times \bar{V}_{c}\times
\bar{V}_{c'})} [1 +\parallel
V\parallel_{\mathcal{E}^{s}(\mathcal{Y}_{t})}]^{2s-3}].
\end{aligned}
\end{equation}
$c_{1}$ being a strictly positive constant independent of $t$.
At the third step We show that $F$ is a contraction of a ball
of $\hat{\mathcal{E}}^{s}(\mathcal{Y}_{T})$ in itself. Let $V^{1}$, $V^{2}$ $\in$ $\mathcal{E}^{s}(\mathcal{Y}_{T})$ such that $u^{1} =F(V^{1})$ and $u^{2}=F(V^{2})$.\\By definition of the map $F$ , it follows that $u^{1}- u^{2}$ satisfies the following  spatio-characteristic lineair problem:
\begin{equation}\label{705tt}
\left\{
                \begin{array}{ll}
\tilde{A}_{V^{1}}\partial^{2}_{\lambda\mu}( u^{1}_{r}-u^{2}_{r})=[\tilde{A}_{V^{2}}-\tilde{A}_{V^{1}}]\partial^{2}_{\lambda\mu}u^{2}_{r}+ \tilde{f}_{V^{1}}-\tilde{f}_{V^{2}} ~~~in ~~~ \mathcal{Y}_{T}\hbox{} \\
u^{1}_{r}-u^{2}_{r}=0 ~~on ~~S_{T}^{1}~;~u^{1}_{r}-u^{ 2}_{r}=0,~~\partial_{0}(u^{1}_{r}-u^{2}_{r})=0 ~~on ~~S_{T}^{ 2}
                \end{array}
              \right.
\end{equation}
Since $u^{1}-u^{2}$ is a solution of the spatio-characteristic linear problem $(\ref{705tt})$, we deduce from inegality $(\ref{4101ksiiisiiiut})$ of  theorem $\ref{40ssytt}$, $\forall$ $t$ $\in$ $]0,T]$, by setting
\begin{equation}\label{706t}
P=[\tilde{A}_{V^{2}}-\tilde{A}_{V^{1}}]\partial^{2}_{\lambda\mu}u^{2}_ {r}+\tilde{f}_{V^{1}}-\tilde{f}_{V^{2}}
\end{equation}
that $\forall$ $t$ $\in$ $]0,T]$,
\begin{equation}\label{707tt}
\parallel u^{1}-u^{2} \parallel_{\mathcal{E}^{s}(\mathcal{Y}_{t})}\leq c_{T} [ t^{\frac{ 1}{2}}\parallel
P\parallel_{\mathcal{K}^{s-1}(\mathcal{Y}_{t})}+\sum\limits_{w=1}^{2}
\parallel
P\parallel_{\mathcal{K}^{s-1}(\Gamma,S^{w},\mathcal{Y})}].
\end{equation}
Based on the fact that:
\begin{equation}\label{708t}
\parallel P\parallel_{\mathcal{K}^{s-1}(\Gamma,S^{w},\mathcal{Y})}\leq \parallel P\parallel_{\mathcal{K}^{s -1}(\mathcal{Y}_{t})},
\end{equation}
The inequality $(\ref{707tt})$ deviates:
\begin{equation}\label{709t}
\parallel u^{1}-u^{2} \parallel_{\mathcal{E}^{s}(\mathcal{Y}_{t})}\leq c_{T}(s) t^{\frac{1}{2}}\parallel P\parallel_{\mathcal{K}^{s-1}(\mathcal{Y}_{t})}.
\end{equation}
 So using the triangular inequality of the norm $\parallel . \parallel_{\mathcal{K}^{s-1}(\mathcal{Y}_{t})}$, we have:
\begin{equation}\label{710t}
\parallel P \parallel_{\mathcal{K}^{s-1}(\mathcal{Y}_{t})}\leq \parallel [\tilde{A}_{V^{2}}-\tilde {A}_{V^{1}}]\partial^{2}_{\lambda\mu}u^{2}_{r} \parallel_{\mathcal{K}^{s-1}(\mathcal{Y}_{t})}+\parallel \tilde{f}_{V^{1}}-\tilde{f}_{V^{2}} \parallel_{\mathcal{K}^{ s-1}(\mathcal{Y}_{t})}.
\end{equation}
According to an inequality analogous to the inequality $(\ref{37ss})$ of theorem $\ref{36ss}$ $d)$ where we take $2\leq r=p=s$, we have: $\forall$ $t$ $\in$ $]0,T]$,
\begin{equation}\label{bxttsd}
\begin{aligned}
\parallel \tilde{f}_{V^{1}}-\tilde{f}_{V^{2}} \parallel_{\mathcal{K}^{s-1}(\mathcal{Y}_ {t})} &\leq & c_{T}(s)\max\limits_{1\leq i\leq 5n} \parallel D_{u_{i}}f\parallel_{C^{2s-3}_ {b}(\mathcal{Y}_{t}\times \bar{V}_{c}\times \bar{V}_{c'})} \times \parallel V^{1}-V^ {2} \parallel_{\mathcal{K}^{s}(\mathcal{Y}_{t})}\\
                                 & &\times [1
+\parallel V^{1}\parallel_{\mathcal{E}^{s}(\mathcal{Y}_{t})}+\parallel V^{2}\parallel_{\mathcal{E}^{ s}(\mathcal{Y}_{t})}]^{2s-3} .
\end{aligned}
\end{equation}Since $u^{2}$ $\in$ $\mathcal{E}^{s}(\mathcal{Y}_{t})$ then $\partial^{2}_{\lambda\mu}u ^{2}$ $\in$ $\mathcal{E}^{s-2}(\mathcal{Y}_{t})$ thus, according to an inequality analogous to the inequality $(\ref{37ss})$ of theorem $\ref{36ss}$ $d)$ where we take $2\leq r=p=s$, we have: $\forall$ $t$ $\in$ $]0,T] $,
\begin{equation}\label{bxttsdd}
\begin{aligned}
\parallel [\tilde{A}_{V^{2}}-\tilde{A}_{V^{1}}] \partial^{2}_{\lambda\mu}u^{2}_ {r} \parallel_{\mathcal{K}^{s-1}(\mathcal{Y}_{t})} &\leq & c_{T}(s)\max\limits_{1\leq i\leq n} \parallel D_{u_{i}}A^{\lambda\mu}\parallel_{C^{2s-3}_{b}(\mathcal{Y}_{t}\times \bar{V }_{c})}\times \parallel V^{1}-V^{2} \parallel_{\mathcal{K}^{s}(\mathcal{Y}_{t})} \\ & & \times[1
+\parallel V^{1}\parallel_{\mathcal{E}^{s}(\mathcal{Y}_{t})}+\parallel V^{2}\parallel_{\mathcal{E}^{ s}(\mathcal{Y}_{t})}]^{2s-3} \times \parallel u^{2}\parallel_{\mathcal{E}^{s}(\mathcal{Y}_{ t})}
\end{aligned}
\end{equation}
Using the relation $(\ref{704})$ for $u^{2}$ in this
last inequality we have:
\begin{equation}\label{bxttsddd}
\begin{aligned}
\parallel [\tilde{A}_{V^{2}}-\tilde{A}_{V^{1}}]\partial^{2}_{\lambda\mu}u^{2}_ {r} \parallel_{\mathcal{K}^{s-1}(\mathcal{Y}_{t})} &\leq & c_{T}(s)\max\limits_{1\leq i\ leq n} \parallel D_{u_{i}}A^{\lambda\mu}\parallel_{C^{2s-3}_{b}(\mathcal{Y}_{t}\times \bar{V }_{c})}\times \parallel V^{1}-V^{2} \parallel_{\mathcal{K}^{s}(\mathcal{Y}_{t})} \\ & & \times[1
+\parallel V^{1}\parallel_{\mathcal{E}^{s}(\mathcal{Y}_{t})}+\parallel V^{2}\parallel_{\mathcal{E}^{ s}(\mathcal{Y}_{t})}]^{2s-3}\\ && \times c_{1} [\parallel \varphi \parallel_{E^{2s-1}(S_{t} ^{1})} + \parallel \varphi \parallel_{E^{2s-2}(\Gamma,S^{2})} +\parallel h \parallel_{E^{s}(S^{2} _{t})} +\parallel g \parallel_{E^{s-1}(S^{2}_{t})} \\ & &+t^{\frac{1}{2}} [ 1
+\parallel V^{2} \parallel_{\mathcal{E}^{s}(\mathcal{Y}_{t})}]^{2s-3}]
\end{aligned}
\end{equation}
By combining the relations $(\ref{bxttsd})$ and $(\ref{bxttsddd})$
in $(\ref{709t})$, the relation $(\ref{709t})$ becomes $\forall$
$t$ $\in$ $]0,T]$:
\begin{equation} \label{748t}
\parallel u^{1}-u^{2} \parallel_{\mathcal{E}^{s}(\mathcal{Y}_{t})} \leq C_{t,V^{1},V ^{2},\max\limits_{1\leq i \leq n}\parallel D_{u_{i}}A^{\lambda\mu } \parallel, \max\limits_{1\leq i\leq 5n }\parallel D_{u_{i}}f\parallel } \times t^{\frac{1}{2}} \times \parallel V^{1}-V^{2} \parallel_{\mathcal{K }^{s}(\mathcal{Y}_{t})}.
\end{equation}
We consider $B^{s}_{R,t}$ the closed part of $\hat{\mathcal{E}}^{s}(\mathcal{Y}_{t})$ defined by:
 \begin{center}
$B^{s}_{R,t}=\lbrace u=(u_{r}) \in \hat{\mathcal{E}}^{s}(\mathcal{Y}_{t}): u_{r}=\bar{u}_{r}=\varphi_{r}~on~S_{t}^{1},~ u_{r}=h_{r},~~\partial_{0} u_{r}=g_{r} ~~on ~~S_{t}^{2},~\parallel u\parallel_{\mathcal{E}^{s}(\mathcal{Y}_{t}) }\leq R \rbrace$
\end{center}
where \begin{center}$R=\max\lbrace \parallel u^{0} \parallel_{\mathcal{E}^{s}(\mathcal{Y}_{t})} , 2c_{1}( \parallel \varphi \parallel_{E^{2s-1}(S_{t}^{1})} + \parallel \varphi \parallel_{E^{2s-2}(\Gamma,S^{2}) } +\parallel h \parallel_{E^{s}(S^{2}_{t})} +\parallel g \parallel_{E^{s-1}(S^{2}_{t}) }) \rbrace$
\end{center}
$c_{1}$ is the constant of the inequality $(\ref{704})$, for which $ u_{r}=\bar{u}_{r}=\varphi_{r}$~sur~ $S^{1}$,~ $u_{r}=h$,~~$\partial_{0}u_{r}=g$ ~~on ~~$S^{2}$ and $u^{ 0}$ solution of the problem $(\ref{715t})$ to ensure that $B^{s}_{R,t}$ is non-empty.
Let $T_{1}$ $\in$ $]0,T]$ such that:
\begin{equation}\label{763t}
c_{1}T_{1}^{\frac{1}{2}} [1
+R ]^{2s-3} < \frac{R}{2}.
\end{equation}
Thus using the relation $(\ref{763t})$, we deduce that the following relation is verified for all $t$ $\in$ $]0,T_{1}]$:
\begin{equation}\label{716t}
\parallel V \parallel_{\mathcal{E}^{s}(\mathcal{Y}_{t})}\leq R \Longrightarrow \parallel u\parallel_{\mathcal{E}^{s}(\mathcal {Y}_{t})}\leq R.
\end{equation}
Let $T_{0}$ $\in$ $]0,T_{1}]$ such that:
\begin{equation}\label{717t}
C_{T_{0},R,\max\limits_{1\leq i \leq n}\parallel D_{u_{i}}A^{\lambda\mu } \parallel, \max\limits_{1\leq i\leq 5n}\parallel D_{u_{i}}f\parallel } T_{0}^{\frac{1}{2}} < \frac{1}{2}
\end{equation} We therefore deduce from Banach's fixed point theorem that there exists a vector function $u$ $\in$ $\hat{\mathcal{E}}^{s}(\mathcal{Y}_{T_{0} })$ such that $u=F(u)$.  At the fourth step, let $u^{1}$, $u^{2}$ $\in$
$\mathcal{E}^{s}(\mathcal{Y}_{T_{0}})$ such that $u^{1}$ and $u^{2}$
verify the hyper-quasilinear spatio-characteristic problem
$(\ref{6999tt})$. Then $ u^{1}-u^{2}$ is solution following  spatio-characteristic linear problem:
\begin{equation}\label{7050t}\left\{
                \begin{array}{ll}
\tilde{A}_{u^{1}}\partial^{2}_{\lambda\mu}( u^{1}_{r}-u^{2}_{r})+Q=
0 ~~~in ~~~ \mathcal{Y}_{T_{0}} \hbox{} \\ u^{1}_{r}-u^{2}_{r}=0 ~~on
~~S_{T}^{1} ~~;~~
u^{1}_{r}-u^{2}_{r}=0,~~\partial_{0}(u^{1}_{r}-u^{2}_{r}) =0 ~~on
~~S_{T}^{2}\end{array}
              \right.
\end{equation} with:
\begin{center}
$Q=[\tilde{A}_{u^{2}}-\tilde{A}_{u^{1}}]\partial^{2}_{\lambda\mu}u^{2}
_{r}+\tilde{f}_{u^{1}}-\tilde{f}_{u^{2}}$.
\end{center}
Thus the solution $u^{1}-u^{2}$ of the spatio-characteristic linear problem $(\ref{7050t})$, verifies the energy inequality $(\ref{4101ksiiisiiiut})$  that $\forall$ $t$ $]0,T_{0}]$,
\begin{equation}\label{7070}
\parallel u^{1}-u^{2} \parallel_{\mathcal{E}^{s}(\mathcal{Y}_{t})}\leq c_{T} [ t^{\frac{ 1}{2}}\parallel Q
\parallel_{\mathcal{K}^{s-1}(\mathcal{Y}_{t})}+\sum\limits_{w=1}^{2} \parallel Q\parallel_{\mathcal{K }^{s-1}(\Gamma,S^{w},\mathcal{Y})}].
\end{equation}
Based on the fact that:
\begin{equation}\label{7080}
\parallel Q \parallel_{\mathcal{K}^{s-1}(\Gamma,S^{w},\mathcal{Y})}\leq \parallel Q\parallel_{\mathcal{K}^{s -1}(\mathcal{Y}_{t})},
\end{equation}
The inequality $(\ref{7070})$ deviates:
\begin{equation}\label{7090}
\parallel u^{1}-u^{2} \parallel_{\mathcal{E}^{s}(\mathcal{Y}_{t})}\leq C_{1} t^{\frac{1 }{2}}\parallel Q\parallel_{\mathcal{K}^{s-1}(\mathcal{Y}_{t})}.
\end{equation}
 Using the standard triangle inequality $\parallel . \parallel_{\mathcal{K}^{s-1}(\mathcal{Y}_{t})}$ we have:
\begin{equation}\label{71000}
\parallel Q \parallel_{\mathcal{K}^{s-1}(\mathcal{Y}_{t})}\leq \parallel [\tilde{A}_{u^{2}}-\tilde {A}_{u^{1}}]\partial^{2}_{\lambda\mu}u^{2}_{r} \parallel_{\mathcal{K}^{s-1}(\mathcal{Y}_{t})}+\parallel \tilde{f}_{u^{1}}-\tilde{f}_{u^{2}} \parallel_{\mathcal{K}^{ s-1}(\mathcal{Y}_{t})}.
\end{equation}According to an inequality analogous to the inequality $(\ref{37ss})$ of theorem $\ref{36ss}$ $d)$ where we take $2\leq r=p=s$, we have: $\forall$ $t$ $\in$ $]0,T_{0}]$,
\begin{equation}\label{n2}
\parallel \tilde{f}_{u^{1}}-\tilde{f}_{u^{2}} \parallel_{\mathcal{K}^{s-1}(\mathcal{Y}_ {t})} \leq C_{2} \parallel u^{1}-u^{2} \parallel_{\mathcal{K}^{s}(\mathcal{Y}_{t})}.
\end{equation} $C_{2}$ being a positive, increasing function of $T_{0}$, $\parallel u^{1}\parallel_{\mathcal{E}^{s}(\mathcal{Y }_{t})}$ , $\parallel u^{2}\parallel_{\mathcal{E}^{s}(\mathcal{Y}_{t})}$, $ \parallel D_{u_{ i}}f\parallel_{C^{2s-3}_{b}(\mathcal{Y}_{t}\times \bar{V}_{c}\times \bar{V}_{c' })}$, independent of $t$.\\{
Since $u^{2}$ $\in$ $\mathcal{E}^{s}(\mathcal{Y}_{t})$ then $\partial^{2}_{\lambda\mu}u ^{2}$ $\in$ $\mathcal{E}^{s-2}(\mathcal{Y}_{t})$ thus, according to an inequality analogous to the inequality $(\ref{37ss})$ of theorem $\ref{36ss}$ $d)$ where we take $2\leq r=p=s$, we have: $\forall$ $t$ $\in$ $]0,T_{ 0}]$,
\begin{equation}\label{n3}
\parallel [\tilde{A}_{u^{2}}-\tilde{A}_{u^{1}}]\partial^{2}_{\lambda\mu}u^{2}_ {r} \parallel_{\mathcal{K}^{s-1}(\mathcal{Y}_{t})} \leq C_{3} \parallel u^{1}-u^{2} \parallel_ {\mathcal{K}^{s}(\mathcal{Y}_{t})}
\end{equation}
$C_{3}$ being a positive, increasing function of $T_{0}$, $\parallel u^{1}\parallel_{\mathcal{E}^{s}(\mathcal{Y}_{t} )}$ , $\parallel u^{2}\parallel_{\mathcal{E}^{s}(\mathcal{Y}_{t})}$, $ \parallel D_{u_{i}}A^ {\lambda\mu}\parallel_{C^{2s-3}_{b}(\mathcal{Y}_{t}\times \bar{V}_{c})}$, independent of $t$ .
By combining the relations $(\ref{n2})$ and $(\ref{n3})$, the inequality $(\ref{71000})$ becomes:
\begin{equation}\label{n4}
\parallel Q \parallel_{\mathcal{K}^{s-1}(\mathcal{Y}_{t})} \leq C_{4} \parallel u^{1}-u^{2} \parallel_ {\mathcal{K}^{s}(\mathcal{Y}_{t})}
\end{equation}$C_{4}$ being a positive, increasing function of $T_{0}$, $\parallel u^{1}\parallel_{\mathcal{E}^{s}(\mathcal{Y}_{t} )}$ ,
 $\parallel u^{2}\parallel_{\mathcal{E}^{s}(\mathcal{Y}_{t})}$,
 $ \parallel D_{u_{i}}f\parallel_{C^{2s-3}_{b}(\mathcal{Y}_{t}\times \bar{V}_{c}\times
 \bar{V}_{c'})}$, $ \parallel D_{u_{i}}A^{\lambda\mu}\parallel_{C^{2s-3}_{b}(\mathcal{ Y}_{t}\times \bar{V}_{c})}$ independent of $t$.\\
By combining the previous relations, the relation $(\ref{n4})$ and $(\ref{7090})$ we have $\forall$ $t$ $\in$ $]0,T_{0}]$ :
\begin{equation}\label{n5}
\parallel u^{1}-u^{2} \parallel_{\mathcal{E}^{s}(\mathcal{Y}_{t})} \leq C_{4} \parallel u^{1} -u^{2} \parallel_{\mathcal{K}^{s}(\mathcal{Y}_{t})}
\end{equation}
$C_{4}$ being of the same nature as $C_{3}$.\\
Using Gronwall's lemma at $(\ref{n5})$ we have:
\begin{center}
$\parallel u^{1}-u^{2} \parallel_{\mathcal{E}^{s}(\mathcal{Y}_{T_{0}})}\leq 0$.
\end{center}
So, in $\mathcal{E}^{s}(\mathcal{Y}_{T_{0}})$ we have $u^{1}=u^{2}$.
\section{Discussion} This work is not only a clarification of the
results of existence and uniqueness for second order linear
hyperbolic spatio-characteristic problems started in $\cite{wa}$ but
also an extension to the hyper-quasilinear hyperbolic
spatio-characteristic problems. We will use this result  to
establish a semi-global existence and uniqueness result \cite{lo}  in
Sobolev spaces with weights
    for second order hyper-quasilinear hyperbolic Goursat problems where the coefficients
     of the derivatives depend linearly on the unknowns variable which we will apply to
     the Einstein vacuum in harmonic gauge.

\end{document}